\documentclass[oneside,english]{amsart}

\usepackage[latin1]{inputenc}
\usepackage{amssymb}
\usepackage{color}
\usepackage{enumerate}
\usepackage{hyperref}
\usepackage{tikz}
\usepackage{pgf}
\usetikzlibrary{arrows,snakes,automata}

\newtheorem{theorem}{Theorem}[section]
\newtheorem{lemma}[theorem]{Lemma}
\newtheorem{prop}[theorem]{Proposition}

\newtheorem{defn}[theorem]{Definition}
\def\F{{\mathcal F}}

\def\N{\mathbb N}

\def \Z {\mathbb Z}
\def \B{\mathcal B}

\def \E{\mathbb E}
\def \R{\mathbb R}

\def \E1{\mathcal E}

\def \d {\partial}
\def \p {\prime}

\def \supp {{\rm supp}}
\def \mR {\mathbb{R}}
\def \mZ {\mathbb{Z}}
\def \mN {\mathbb{N}}

\makeatletter

\title{One dimensional Markov random fields, Markov chains and  Topological Markov fields}
\begin{document}
\maketitle

%bhm
\noindent {\small{
\author{Nishant Chandgotia, University of British Columbia}\\
\author{Guangyue Han, University of Hong Kong}\\
\author{Brian Marcus, University of British Columbia}\\
\author{Tom Meyerovitch, University of British Columbia}\\
\author{Ronnie Pavlov, University of Denver}
}} \bigskip

\section{Introduction}

A one-dimensional Markov chain is defined by a one-sided, directional conditional independence property, and a process is Markov in the forward direction if and only if it is Markov
in the backward direction.  In two and higher dimensions, this property
is replaced with a conditional independence property that is not associated with a particular direction. This leads to the notion of a Markov random field (MRF).

Of course, the definition of MRF
%bhm
makes sense
in one dimension as well (see Section~\ref{background}),
but here the conditional independence is a two-sided property, and
this is not the same as the Markov chain property. It is well known
that any one-dimensional Markov chain is an MRF. However, the
converse is not true: there are counter-examples for non-stationary,
finite-valued processes and for stationary countable-valued
processes. The converse does hold (and this is well known) for
finite-valued stationary MRF's that either have full support or
satisfy a certain mixing condition. In this paper we show that any
one-dimensional stationary, finite-valued MRF is a Markov chain,
without any mixing condition or condition on the support.

Our proof makes use of two properties of the support $X$ of a finite-valued stationary
MRF: 1) $X$ is non-wandering (this is a property of the support of any finite-valued stationary
process) and 2) $X$ is a topological Markov field (TMF) (defined in Section~\ref{background}).
The latter is a new property that sits in between the classes of shifts of finite type and
sofic shifts, which are well-known objects of study in symbolic dynamics~\cite{LM}.
Here, we develop the TMF property in one dimension, and we will develop this
property in higher dimensions in a
future paper.

While we are mainly interested in discrete-time finite-valued stationary MRF's,
in Section~\ref{continuous} we also consider continuous-time, finite-valued
stationary MRF's, and show that these are (continuous-time) Markov chains as well.

\section{Background}
\label{background}

\subsection{Basic Probabilistic Concepts}

Except for Section~\ref{continuous} (where we consider
continuous-time processes) by a stochastic process we mean a
discrete-time, finite-valued process defined by a probability
measure $\mu$ on the measurable space
$(\Sigma^\mathbb{Z},\mathcal{B})$, where $\Sigma$ is a finite set
(the \emph{alphabet}) and $\mathcal{B}$ is the product Borel
$\sigma$-algebra, which is generated by the \emph{cylinder sets}
$$
[a_1,\ldots,a_n]_j := \{ x \in \Sigma^\mathbb{Z}~:~ x_{k+j}=a_k \mbox{ for } k = 1, \ldots, n\},
$$
for any $a_i \in \Sigma$ and $j \in \mathbb{Z}$.

Throughout this paper, we will often use the shorthand notation:
$$
\mu(a_{i_1}, \ldots, a_{i_n}) = \mu(\{x \in \Sigma^\mathbb{Z}~:~
x_{i_k}=a_{i_k} \mbox{ for } k = 1, \ldots ,n\})
$$
and similarly for conditional measure:
\begin{eqnarray*}
&\mu(a_{i_1}, \ldots, a_{i_n}~|~b_{j_1}, \ldots, b_{j_m} )&\\
=&\mu(\{x \in \Sigma^\mathbb{Z}~:~ x_{i_k}=a_{i_k}, ~ k = 1 \ldots n \}~|~\{x_{j_\ell}=b_{j_\ell}, ~ \ell = 1 \ldots m \}).&
\end{eqnarray*}
In particular,
$$
\mu(a_1, \ldots, a_n) = \mu([a_1, \ldots, a_n]_0).
$$
Also, for $x\in \Sigma^\mathbb{Z}$ and $a \leq b \in \mathbb{Z}$, we define $x_{[a,b]} = x_a
x_{a+1} \ldots x_b$.

A stochastic process is {\em stationary} if it satisfies
\begin{eqnarray*}
\mu(a_1, a_2, \ldots, a_n)=\mu(a_{j+1}, a_{j+2}, \ldots, a_{j+n})
\end{eqnarray*}
for all $j \in \Z $.

A \emph{Markov chain} is a stochastic process which satisfies the
usual Markov condition:
$$\mu(a_0, \ldots, a_n  \mid a_{-N}, \ldots, a_{-1} ) = \mu(a_0,\ldots,a_n  \mid a_{-1})$$
whenever  $\mu(a_{-N},\ldots,a_{-1}) > 0$.

A \emph{Markov random field (MRF)} is a stochastic process $\mu$ which satisfies
$$
\mu(a_0, \ldots, a_n \mid a_{-N}, \ldots, a_{-1}, a_{n+1},\ldots, a_{n+M})
$$
$$
= \mu(a_0, \ldots, a_n \mid a_{-1}, a_{n+1}),
$$
whenever  $\mu(a_{-N}, \ldots, a_{-1}, a_{n+1},\ldots, a_{n+M}) > 0$.

Note that the Markov chains that we have defined here are first
order Markov chains. Correspondingly, our MRF's are first order in a
two-sided sense.  One can consider higher order Markov chains and
MRF's, but these can be naturally recoded to first order processes, and all of our
results easily carry over to higher order processes.

%bhm
More generally, a Markov random field can be defined on an
undirected graph $\mathcal{G}=(V,E)$, where $V$ is a countable set
of vertices and $E$, the set of edges, is a set of unordered pairs
of distinct vertices. Specifically, an MRF on $\mathcal{G}$  is a
probability measure $\mu$ on $\Sigma^V$ for which $\mu([a_F] \mid
[b_{\partial F}] \cap [c_{ G}]) = \mu([a_F] \mid [b_{\partial F}])$
whenever $G$ and $F$ are finite subsets of $V$, $G \cap F =
\emptyset$ and $\mu( [b_{\partial F}] \cap [c_{ G}] ) > 0$; here,
the notation such as $[a_F]$ means a configuration of letters from
$\Sigma$ on the subset $F \subset V$, and $\partial F$ (the boundary
of $F$) denotes the set of $v \in V\setminus F$ such that
$\{u,v\}\in E$ for some $u \in F$. It is not hard to see that when
the graph is the one-dimensional integer lattice, this agrees with
the definition of MRF given above.

The following is well known.  We give a proof for completeness.

\begin{prop}
\label{MC_is_MRF}
Any Markov chain (stationary or not) is an MRF.
\end{prop}

\begin{proof}

$$\mu(a_0,\ldots,a_n, a_{n+1},\ldots, a_{n+M}  ~|~   a_{-N},\ldots,a_{-1})
$$
$$
=
\mu(a_{n+1}, \ldots,  a_{n+M} ~|~   a_{-N},\ldots,a_{-1})
\mu(a_0,\ldots,a_n \mid a_{-N},\ldots,a_{-1}, a_{n+1},\ldots, a_{n+M})
$$
Thus,
$$
\mu(a_0,\ldots,a_n \mid a_{-N},\ldots,a_{-1}, a_{n+1},\ldots, a_{n+M})
$$
$$
=
\frac{ \mu(a_0,\ldots,a_n, a_{n+1},\ldots, a_{n+M} ~|~   a_{-N},\ldots,a_{-1})}
{\mu( a_{n+1},\ldots, a_{n+M}  ~|~   a_{-N},\ldots,a_{-1}) }
$$
$$
=
\frac{ \mu( a_0,\ldots,a_n , a_{n+1},\ldots, a_{n+M} ~|~   a_{-1})}
{\mu( a_{n+1},\ldots, a_{n+M} ~|~ a_{-1})}
$$
the latter by Markovity. Since the last expression does not involve $a_{-N},\ldots,a_{-2}$, we have
$$
\mu(a_0,\ldots,a_n \mid a_{-N},\ldots,a_{-1}, a_{n+1},\ldots, a_{n+M}) =
\mu(a_0,\ldots,a_n \mid a_{-1}, a_{n+1},\ldots, a_{n+M})
$$
Since the reverse of a Markov process is Markov, we then have, by symmetry,
$$
\mu(a_0,\ldots,a_n \mid a_{-1}, a_{n+1},\ldots, a_{n+M}) =
\mu(a_0,\ldots,a_n \mid a_{-1}, a_{n+1})
$$
Combining the previous two equations, we see that the MRF property holds.
\end{proof}

See~\cite[Corollary 11.33]{Georgii} for an example of a fully
supported stationary MRF on a countable alphabet that is not a
Markov chain. It is easy to construct examples of non-stationary
MRF's that are not Markov chains (see the remarks immediately
following Proposition~\ref{prop:not markov}).

\subsection{Symbolic Dynamics}
\label{symbolic}

In this section, we review concepts from symbolic dynamics.  For more details, the reader may consult~\cite[Chapters 1-4]{LM}.

Let $\Sigma^*$ denote the collection of words of finite length over a finite alphabet $\Sigma$.
For $w \in \Sigma^*$, let $|w|$ denote the length of $w$.

Let $\sigma$ denote the \emph{shift map}, which acts on a
bi-infinite sequence $x \in \Sigma^{\mathbb{Z}}$ by shifting all
symbols to the left, i.e.,
$$
(\sigma(x))_n=x_{n+1} \mbox{ for all } n.
$$
A subset $X$ of $\Sigma^{\mathbb{Z}}$ is \emph{shift-invariant} if
$\sigma(x) \in X$ for all $x \in X$. A \emph{subshift} or a
\emph{shift space} $X \subset\Sigma^\mathbb{Z}$ is a shift-invariant
set which is closed with respect to the product topology on
$\Sigma^\mathbb{Z}$. Note that $\Sigma^\mathbb{Z}$ itself is a shift
space and is known as the {\em full shift}. There is an equivalent
way of defining shift spaces by using forbidden blocks: for a subset
$\F$ of $\Sigma^*$, define
\begin{equation*}
X_{\F}=\{x \in \Sigma^{\mathbb{Z}} \;|\: (x_i x_{i+1}\ldots x_{i+j}) \notin \F \text{ for all } i \in \Z \text{ and } j \in \N \cup\{0\}\}.
\end{equation*}
Then $X$ is a shift space iff $X = X_{\F}$ for some $\F$.
% $X$
%is the full shift iff $X=\Sigma_{\varnothing}$. (\textbf{check
%here}).

The \emph{language} of a shift space $X$ is
$$
B(X)= \bigcup_{n=1}^\infty B_n(X)
$$
where
$$
B_n(X)=\{w \in \Sigma^*~:~ \exists x \in X  \mbox{ s.t. } (x_1\ldots x_n) = w \}.
$$

A \emph{sliding block code} is a continuous map $\phi$ from one
shift space $X$, with alphabet $\Sigma$, to another $Y$, with
alphabet $\Sigma'$, which commutes with the shift: $\phi \circ
\sigma = \sigma \circ \phi$. The terminology comes from the
Curtis-Lyndon-Hedlund Theorem which characterizes continuous
shift-commuting maps as generated by finite block codes, namely:
there exist $m,n$ and a map $\Phi: B_{m+n+1}(X)\longrightarrow
\Sigma'$ such that
$$(\phi(x))_i= \Phi(x_{i-m }x_{i-m+1}\ldots x_{i+n})$$
If $m=0=n$, then $\phi$ is called a \emph{1-block map}.
A \emph{conjugacy} is a bijective sliding block code, and a
\emph{factor map} or (\emph{factor code}) is a surjective sliding block code.

A\emph{ shift of finite type (SFT)} is a shift space $X = X_\F$
where $\F$ can be chosen finite. An SFT is called \emph{$k$-step} if
$k$ is the smallest positive integer such that $X = X_\F$ and $\F
\subset \Sigma^{k+1}$.  A 1-step SFT is called a \emph{topological
Markov chain (TMC)}. Note that a TMC can be characterized as the set
of all bi-infinite vertex sequences on the directed graph with
vertex set $\Sigma$ and an edge from $x$ to $y$ iff $xy \not\in\F$.
TMC's were originally defined as analogues of first-order Markov chains,
where only the transitions with strictly positive transition probability
are prescribed~\cite{pa64}.

The most famous TMC is the {\em golden mean shift} defined over the
binary alphabet by forbidding the appearance of adjacent 1's,
equivalently $X = \Sigma_{\{11\}}$.  The corresponding graph is
shown in Figure~\ref{gm}.

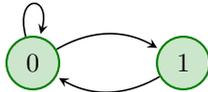
\begin{figure}
\caption{golden mean shift}
\label{gm}
\begin{center}
\begin{tikzpicture}[->,>=stealth,shorten >=1pt,auto,node distance=2cm, semithick]
        \tikzstyle{every state}=
        [%
          fill=green!50!black!20,%
          draw=green!50!black,%
          minimum size=7mm,%
          circle,%
          thick%
        ]
        \node[state] (A) {$0$};
        \node[state] (B) [right of=A] {$1$};
    \path (A) edge[bend left, above=-2pt]  (B)
          (B) edge[bend left, above=-2pt]  (A);
        \path (A) edge[loop above] (A);
\end{tikzpicture}
\end{center}
\end{figure}

%
%by a $|\Sigma| \times |\Sigma|$ matrix, whose entry indexed by $(x,
%y)$ (here $x, y \in \Sigma$) is $1$ if $xy \in \Sigma^*$, and $0$
%otherwise.

Just as MRF's and Markov chains can be recoded to first order
processes, any SFT can be recoded to a TMC; more precisely, any SFT
is conjugate to a TMC.

A shift space $X$ is \emph{non-wandering} if whenever $ u \in B(X)$, there exists a word $v$ such that $uvu \in B(X)$.
The \emph{support} of a stationary process $\mu$ on $(\Sigma^\mathbb{Z},\mathcal{B})$ is
the set
$$\supp(\mu) = \Sigma^\mathbb{Z} \setminus \bigcup_{[a]_n \in \mathcal{N}(\mu)}[a]_n,$$
where $\mathcal{N}(\mu)$ is the collection of all cylinder sets with
$\mu([a]_n)=0$. TMC's are exactly the set of shift-invariant sets
that can be the support of a (first-order) Markov chain. The reason
for our interest in the non-wandering property is that the support
of any stationary measure is non-wandering; this follows immediately
from the Poincare  Recurrence Theorem  (see~\cite{Pet}).

A shift space is {\em irreducible} if whenever $u,v \in B(X)$,
there exists $w$ such that $uwv \in B(X)$. Clearly any irreducible shift space is non-wandering.
By using the decomposition of a non-negative matrix into irreducible components, it is easy to see that
a TMC is non-wandering
iff it is the union of finitely many irreducible TMC's on disjoint
alphabets. For a subshift $X$, the period of $x \in X $ is $\min \{i \in \N\; |\;\sigma^i(x)=x\}$.
The \emph{period} of X is defined as the greatest common divisor of the periods of elements of X.
For any irreducible TMC of period $p$, one can partition $\Sigma$ into
$\Sigma_0, \Sigma_1, \ldots, \Sigma_{p-1}$ such that if $x \in X, x_0 \in  \Sigma_j$ then
$x_i \in \Sigma_{i + j \pmod p}$. The $\Sigma_i$'s are called the {\em cyclically moving subsets}.

In the full shift, any symbol can appear immediately after any other
symbol. Irreducible TMC's have a related property, described as
follows. The \emph{index of primitivity} of a TMC $X$ with period
$p$ is the smallest positive integer  $t$ such that for any $a \in
\Sigma_i, b \in \Sigma_j$, there exists  $x \in X$ such that $x_1
=a, x_{tp+j-i+1}=b$. Every irreducible TMC has a finite index of
primitivity. The terminology comes from matrix theory where, for a
primitive matrix $A$,  the index of primitivity is understood to be
the smallest positive integer such that $A^n$ is strictly positive,
entry by entry. For example, the golden mean shift is irreducible,
its period is 1 and its index of primitivity is 2.

A \emph{sofic shift} is a shift space which is a factor of an SFT.
By a recoding argument, it can be proved that any sofic shift $X$ is
a 1-block factor map of a TMC. This amounts  to saying that a shift
space $X$ is sofic iff there is a finite directed graph, whose
vertices are labelled by some finite alphabet, such that $X$ is the
set of all label sequences of bi-infinite walks on the graph. Such a
labelling is called a \emph{presentation}. The most famous sofic
shift is the even shift defined by forbidding sequences of the form
$10^{2n+1}1$ for all $n \in \mN$.  A presentation is shown in
Figure~\ref{even}.

\begin{figure}
\caption{even shift} \label{even}
\begin{center}
\begin{tikzpicture}[->,>=stealth,shorten >=1pt,auto,node distance=2cm, semithick]
        \tikzstyle{every state}=
        [%
          fill=green!50!black!20,%
          draw=green!50!black,%
          minimum size=7mm,%
          circle,%
          thick%
        ]
        \node[state] (A) {$0$};
        \node[state] (B) [right of=A] {$0$};
        \node[state] (C) [left of=A] {$1$};
    \path (A) edge[bend left, above=-2pt]  (B)
          (B) edge[bend left, above=-2pt]  (A);
        \path (C) edge[loop above] (C);
     \path (C) edge[bend left, above=-2pt] (A);
     \path (B) edge[bend left, above=-6pt] (C);
\end{tikzpicture}
\end{center}
\end{figure}
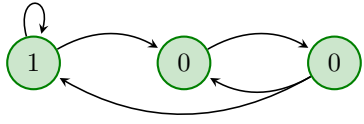

The following is one of the many useful characterizations of sofic
shifts. Let $X$ be a shift space.
 For all $w \in B(X)$, the {\em follower set} of $w$ is defined as
$$
F(w)=\{y \in B(X)~:~ wy \in B(X) \}.
$$
The collection of follower sets is denoted by
$$
F(X)=\{F(w)~:~ w \in B(X) \}.
$$
A shift space $X$ is sofic if and only if $F(X)$ is finite.
Similarly, one can define predecessor sets: $P(w)=\{y \in B(X)~:~
yw \in B(X) \}$, and $P(X)$ denotes the collection of all
predecessor sets. A shift space $X$ is sofic if and only if $P(X)$
is finite.

One can also work with follower sets of left-infinite sequences. Let
$B_\infty(X)$ denote the set of all left-infinite sequences $x^-$
for which there exists a right-infinite sequence $x^+$ such that
$x^-x^+ \in X$.  Let
$$
F_\infty(x^-)=\{y \in B(X)~:~ x^-y \in B_\infty(X) \}.
$$
Note that $F_\infty(x^-)$ is the decreasing intersection of
$\{F(x_{[-n,-1]})\}_n$.  It follows that $X$ is sofic iff the
collection of all $F_\infty(x^-)$ is finite~\cite[Lemma
2.1]{Krieger}.

Finally, we state a useful result that is probably well known, but we do
not know of an explicit reference.  So, we give a proof for
completeness.

\begin{lemma}\label{lem:sofic_non-wandering}
Let $X$ be a shift space. If $X$ has dense
periodic points, then $X$ is non-wandering.  The converse
is true if $X$ is sofic.
\end{lemma}
\begin{proof}

Assume that the periodic points are dense in
$X$. Let $u \in B_n(X)$ for some $n \in \N$. Then there
exists a periodic point $x \in X$ such that $x_{[1,n]}=u$. Since
$x$ is periodic, there exists $r\in \N$ and a word $v$ such that
$x_{[1,r]}=uvu$. Thus
$uvu \in B(X)$, and so X is non-wandering.

For the converse, assume that $X$ is sofic and is non-wandering. Let
$Y$ be a TMC and $\phi:Y  \longrightarrow X$ a $1$-block factor map.
% Write
%$\phi= \Phi^{[0,0]}_{\infty}$, where $\Phi$ is a map from $B_1(Y)$
%to $B_1(X)$.

Consider a word $u \in B_n(X)$ for some $n \in \N$.
Since $X$ is non-wandering, for any $M$,
there exist $v_1 \ldots v_M \in B(X)$ such that $w = uv_1uv_2 \ldots v_Mu \in B(X)$.
Let $M>|B_n(Y)|$. Take $z\in X$ such that for some interval $I$,  $z_{I} = w$.
Then there exist $r_1 < r_2 < \ldots < r_M$ such that
$z_{[r_t, r_{t} +n-1]}= u$ for $t=1, \ldots, M$.
Let $y \in Y $ such that $\phi(y)=z$.
Since $M>|B_n(Y)|$ we can find
$h < k \in \N$ such that
\begin{equation*}
y_{[r_h,r_h+n-1]}=y_{[r_k,r_k+n-1]}.
\end{equation*}
Consider $y^\p$ defined by
\begin{eqnarray*}
y^\p_{[r_h,r_k-1]}&=&y_{[r_h,r_k - 1]}\\
\text{and }y^\p_{t}&=&y^\p_{t+r_k-r_h}
\end{eqnarray*}
for all $t \in \Z$. Clearly $y^\p$ is periodic.  And $y^\p \in   Y$ since
$Y$ is a TMC.
Also,
\begin{equation*}
\phi(y^{\p})_{{[r_h,r_h+n-1]}} =\phi(y)_{{[r_h,r_h+n-1]}}=u
\end{equation*}
Since $u \in B(X)$ was arbitrary, this proves that periodic points
are dense in $X$.
\end{proof}

\section{Topological Markov Fields}

\begin{defn}A shift space $X \subset \Sigma^\mathbb{Z}$ is a \emph{topological Markov field (TMF)}
if whenever $v,x \in \Sigma$,
$|z|=|w|$, and  $uvwxy,vzx \in B(X)$, then $uvzxy \in B(X)$.
\end{defn}

We have defined TMF's as ``1-step'' objects, in that $v,x$ are required to be
letters of the alphabet, i.e., words of length one.   One can
naturally extend the definition to ``$k$-step'' objects, by requiring
$v,x$ to be words of the same length $k$, and results of this section
extend easily to this class.

The defining property for TMF's is equivalent to another property,
which appears stronger, and is suitable for generalization to higher
dimensions.
%bhm
For this, recall that the boundary of $C$, denoted $\d C$, denotes
the set of integers in $\mZ \setminus C$ that are adjacent to an
element of $C$.

\begin{prop}
\label{general}
A shift space $X$ is a TMF if and
only if it satisfies the following condition:

~~~ For all $x,y \in X$ and finite $C \subset \mZ$ such that $x=y$ on $\d C$, the point $z \in \Sigma^\Z$ defined by
\begin{equation*}
z=\begin{cases}x \text{ on } C \cup \d C\\
y \text{ on } (C \cup \d C)^c\end{cases}
\end{equation*}
belongs to $X$.
\end{prop}
\begin{proof} The ``if'' direction is trivial.  For the ``only if'' direction, we use the fact that any finite subset of $\Z$ is a
disjoint union of finitely many intervals (of integers). If $C$ consists of only one interval,
then we get the condition immediately from the definition of TMF.  Now, proceed by induction on the number of intervals.
\end{proof}

The reason for our interest in TMF's is the following simple result.

\begin{lemma}\label{lem:supp_MRF_TMF}
The support of a stationary MRF is a TMF.
\end{lemma}
\begin{proof}
Let $\mu$ be an MRF, $X = \supp(\mu)$ and $uvwxy,vzx \in B(X)$ with
$|z|=|w|$ and $v,x \in \Sigma$. By definition of $X$,
%bhm
$\mu(uvwxy)>0$, and so  $\mu([uv]_0 \cap [xy]_{|uvw|}) > 0$. Since
$\mu$ is an MRF,
%bhm
\begin{equation}
\label{uv} \mu([z]_{|uv|} \mid [uv]_0 \cap [xy]_{|uvw|}) =
\mu([z]_{|uv|} \mid [v]_{|u|} \cap [x]_{|uvw|}).
\end{equation}
Since by the definition of $X$, $\mu(vzx)>0$, it follows that  the
right hand side of (\ref{uv}) is positive. Thus $\mu(uvzxy)>0$, and
so $uvzxy \in B(X)$ as desired.
\end{proof}
For a shift space $X$ and $w \in B(X)$, let
$$
C(w) = \{(x,y)~:~ xwy \in B(X)\} \mbox{ and } C(X) = \{C(w): w \in B(X)\}.
$$
The following is a simple restatement of the definition of TMF in terms of the sets $C(w)$:

\begin{prop}
\label{TMF_charac} A  shift space $X$ is a TMF iff for all $n \in \N$ and for all $w,u \in B_n(X)$,
\begin{equation}
\label{condition} (w_1=u_1 \mbox{ and } w_n =u_n)
\Rightarrow  C(w) = C(u).
\end{equation}
\end{prop}

\begin{prop}\label{prop:TMF_sofic}
Any TMF is sofic.
\end{prop}
\begin{proof}

Let $X \subset \Sigma^{\mathbb{Z}}$ be a shift space that is not sofic.
We will prove that $X$ is not a TMF.

Since $X$ is not sofic, there are infinitely many left-infinite
sequences with distinct follower sets.  Thus there exist distinct
left-infinite sequences $w^1, w^2, \ldots, w^{|\Sigma|^2+1}$ with
distinct follower sets. Note
\begin{equation*}
F_\infty(w^i)=\bigcap_{n \in \N} F(({w^i})_{[-n,-1]})
\end{equation*}
 and $F(({w^i})_{[-n-1,-1]})\subset F(({w^i})_{[-n,-1]})$ for all $1\leq i \leq {|\Sigma|^2+1}$ and $n\in\N$.
 If, for each $n$, the $F(({w^i})_{[-n,-1]})$ are not distinct, then
 there exist $i_1\neq i_2$ such that
 $F_{\infty}({w^{i_1}})=F_{\infty}({w^{i_2}})$,
contradicting the assumption. Therefore there exists an $n_0\in \N$
such that the $F(({w^i})_{[-n_0,-1]})$ are all distinct. Hence we
can choose $u, u^\p \in B_{n_0}(X)$ such that $u_1= u^\p_1$ and
$u_{n_0}=u^\p_{n_0}$ but $F(u)\neq F(u^\p)$. Hence there exists $b
\in B(X)$ such that exactly one of $ub$ and $u^\p b$ is an element
of $B(X)$. It follows that there exists $a \in B(X)$ such that
exactly one of $aub$ and $au^\p b$ is an element of $B(X)$.
Therefore $C(u)\neq C(u^\p)$. By Proposition~\ref{TMF_charac}, $X$
is not a TMF.

\end{proof}

It is clear that any TMC is a TMF. The following example shows that
a TMF need not be a TMC. In fact, this TMF is not even an SFT. This is an
elaboration of an example given in~\cite{Dob}. Let
\begin{equation*}
X_{not}=\{0^\infty, 0^\infty1^\infty, 1^\infty, 1^\infty02^\infty, 2^\infty\};
\end{equation*}
to clarify the notation, $1^\infty0 2^\infty$ refers to the point $x$ such that
\begin{eqnarray*}
x_i=\begin{cases}1 \text{ if } i<0\\
0 \text{ if } i=0\\
2 \text{ if } i>0\end{cases}
\end{eqnarray*}
and all its shifts.\\
\begin{prop}\label{prop:not markov}
$X_{not}$ is a TMF but not an SFT (and in particular is not a TMC).
\end{prop}
\begin{proof}
$X_{not}$ is not an SFT since for all $n$, $01^n, 1^n0 \in
B(X_{not})$ but $01^n0 \notin B(X_{not})$.

We will check that the restriction of any configuration on the positions $0$ and $n+1$ uniquely determines the configuration
on either $[1,n]$ or $[0,n+1]^c$.  This clearly implies condition
(\ref{condition}) and thus, by  Proposition~\ref{TMF_charac},
$X_{not}$ is a TMC.

To see this, first observe
that for $n>2$,
\begin{equation*}
B_n(X_{not})=\{0^n, 0^k1^{n-k}, 1^n, 1^m02^{n-m-1}, 1^{n-1}0,
02^{n-1},2^n\;|\; 0 < k < n, ~0 < m < n-1 \}.
\end{equation*}
Among these, the only pairs of distinct words with the same length that begin and end
with the same symbol are $0^k1^{n-k}, 0^{k^\p}1^{n-{k^\p}}$ and
$1^m02^{n-m-1}, 1^{m^\p}02^{n-{m^\p}-1}$. Now, observe that
%
%It is clear that (\ref{condition}) is satisfied for each of
%$\{0^n\}$, $\{1^n\}$, $\{1^{n-1}0\}$, $\{02^{n-1}\}$,$\{2^n\}$,
%. Furthermore, we verify that
\begin{eqnarray*}
C( 0^k1^{n-k})=&C( 0^{k^\p}1^{n-{k^\p}})&=\{(0^i,1^j)\}\\
C(1^m02^{n-m-1})=&C(1^{m^\p}02^{n-{m^\p}-1})&=\{(1^i,2^j)\}
\end{eqnarray*}
for $0 < k, k^\p<n$ and $0 < m,m^\p<n-1$.% Therefore,
%(\ref{condition}) is satisfied, and thus $X_{not}$ is a TMF.

\end{proof}

Since $X_{not}$ is countable, any strictly positive countable
probability vector defines a measure whose support is $X_{not}$. Any
such measure is an MRF because any valid configuration in $X_{not}$
on the positions $0$ and $n+1$ uniquely determines the configuration
on $[1,n]$ or $[0,n+1]^c$.
%(this can be checked directly).
Thus,
there exist non-stationary finite-valued MRF's which are not Markov
chains of any order.

Now we will prove that there is a finite procedure for checking
whether a sofic shift is a TMF. The following characterisation of
sofic shifts will be used.
\begin{prop}
\label{soficCX}
A  shift space $X$ is sofic iff $|C(X)| < \infty$.
\end{prop}

\begin{proof}

``If:'' For each $F \in F(X)$, fix some $w_F \in B(X)$ such that
$F(w_F) = F$,  Now, consider the map $\Psi: F(X) \rightarrow C(X)$,
defined by $\Psi(F) = C(w_F)$.  If $F \ne F'$, then $F(w_F) \ne
F(w_{F'})$ and thus $C(w_F) \ne C(w_{F'})$.  Thus, $\Psi$ is 1-1,
and so $F(X)$ is finite.

``Only If:'' For a follower set $F \in F(X)$ and a word $w \in F$,
let $F_w = \{y: wy \in F\}$. Note that $F_w$ is a follower set. We
claim that
\begin{equation}
\label{number} C(w) = \{(x,y):   \mbox{ there exists } F \mbox{ s.t.
} w \in F, y \in F_w, \mbox{ and } x \in \cap_{z \in F} P(z) \}.
\end{equation}

To see this, first note that  if $(x,y) \in C(w)$, then $wy \in F =
F(x)$. Then $w \in F$, $y \in F_w$ and for all $z \in F = F(x)$, $xz
\in B(X)$,  and so $x \in P(z)$.

Conversely, if  $w \in F$ and $y \in F_w$, then $z = wy \in F$. If
$x \in \cap_{z \in F} P(z)$, then taking $z = wy$, we see that $xwy
\in B(X)$, and so $(x,y) \in C(w)$.  This establishes
(\ref{number}).

By (\ref{number}), we see that $C(w)$ is uniquely determined by the
set
$$
\{ (\cap_{z \in F} P(z), F_w ):~  F \mbox{ is a follower set that
contains }w \}.
$$
Thus, $|C(X)|$ is upper bounded by  the number of functions whose
domain is a subset of $F(X)$ and whose range is subset of $2^{P(X)}
\times F(X)$ and is therefore finite (here, for a given word $w$,
the domain $D$ is the collection of follower sets that contain $w$
and the function is: for $F \in D$,
%bhm
$g(F) = (\{P(z)\}_{z\in F}, F_w))$.

\end{proof}

Now we can introduce the procedure.

\begin{theorem}\label{thm:check_TMF}
There is a finite algorithm to check whether a given sofic shift
$X$ is a TMF (here, the input to the algorithm is
a labelled finite directed graph presentation of $X$).
\end{theorem}
\begin{proof}

By combining the next two lemmas, we will see that
condition (\ref{condition}) can be decided by checking
words of a length bounded by an explicit function of a
presentation of
$X$.

\begin{lemma}\label{lemma:boundwordsize}
Let $n>|C(X)|^2$ and $w,u \in B_n(X)$ such that $w_1=u_1$ and
$w_n= u_n$. Then there exists $r\le |C(X)|^2$, $w^*, u^* \in
B_r(X)$ such that $w^*_1= u^*_1$, $w^*_r= u^*_r$,
$$
C(w^*) = C(w)
\mbox{ and } C(u^*) = C(u).
$$
\end{lemma}

\begin{proof}
We claim that if $C(a) = C(c)$ for some $a, c \in B(X)$ then for any $b \in B(X)$ such that $ab, cb \in B(X)$, $C(ab) =C(cb)$.
To see this, observe:
\begin{eqnarray*}
C(ab) &=& \{(x,y) : xaby \in B(x)\}\\& =& \{(x,y) : (x,by) \in C(a)\}\\& =&
\{(x,y) : (x,by) \in C(c)\}\\ &=& C(cb).
\end{eqnarray*}

Consider the set $\{(C(u'),C(w'))\}$ such that  $u'$ and $w'$ are
proper prefixes of $u$ and $w$ (respectively) and of the same length.
This set has size at most $|C(X)|^2$.  Since  $n > |C(X)|^2$, there
are distinct pairs $(u',w')$ and $(u'',w'')$, where $u'$ and $u''$ are prefixes of $u$, $w'$ and $w''$ are prefixes of $w$ such that
$$
|u'|=|w'|, \qquad |u''|=|w''|, \qquad C(u')=C(u''), \qquad C(w')=C(w'').
$$
We may assume that $|u'| < |u''|$.

Define words $h,k$ by: $u=u''h$ and $w=w''k$. Since
$$
C(u')=C(u'') \mbox{ and } C(w')=C(w''),
$$
we have
$$
C(u'h)=C(u''h) = C(u) \mbox{ and }
C(w'k)=C(w''k) = C(w).
$$
%bhm
If $u^*=u'h$ and $w^*=w'k$ have length at most $|C(X)|^2$, we are
done; if not, inductively apply the same argument to $u^*,w^*$
instead of $u,w$.
\end{proof}

%To complete the proof of the theorem, by virtue of the preceding
%lemma, it suffices to show that

%We now show that to
%decide if $C_w = C_u$ for a single pair $u,w$ of the same length $n$
%such that $w_1=u_1$ and $w_n= u_n$, one need only consider elements $(x,y)$

Let
$$
C_m(w) = \{(x,y)~:~ xwy \in B(X), |x|, |y| \le m\}.
$$

\begin{lemma}\label{lemma:boundboundarysize}
%bhm
Let $m = \max\{|P(X)|,|F(X)|\}$ and $w,u \in B(X)$ such that $|w| =
|u|$, $w_1=u_1$ and $w_{|w|}= u_{|u|}$. Then
$$
C(w) = C(u) \mbox{ iff } C_m(w) = C_m(u).
$$

\end{lemma}
\begin{proof}
Assume that
$C_m(w) = C_m(u)$.
Assume that $a w b \in B(X)$. It suffices to prove $a u b \in
B(X)$.

 If $|a|, |b|\leq m$ then there is nothing to prove. Suppose instead $|a|>m$.
 By the choice of $m$, there exists
%bhm
 $1 \le i < i^\p \le |a|$ such that
\begin{eqnarray*}
F(a_1 \ldots a_i)&=&F(a_1 \ldots a_{i^\p}).
\end{eqnarray*}
Let $a^\p= a_1 \ldots a_i a_{i^\p+1} \ldots a_{|a|}$. Then $a^\p w b
\in B(X).$ And $a u b \in B(X)$ iff $a^\p u b\in B(X)$. Since
$|a^\p|<|a|$, by induction on $|a|$, we can assume $|a^\p|\leq m$.
By a similar argument, we can find $b^\p$ such that $|b^\p|\leq m$,
$a^\p w b^\p \in B(X)$.  And $a u b \in B(X)$ iff  $a^\p u b^\p
\in B(X)$.
Since $C_m(w) = C_m(u)$,
$a^\p u b^\p \in B(X)$.  Therefore, $a u b \in B(X),$ as desired.
\end{proof}

{\em Proof of Theorem~\ref{thm:check_TMF}:} By Lemmas
\ref{lemma:boundwordsize} and \ref{lemma:boundboundarysize}, $X$ is
a TMF if and only if for all $r \le |C(X)|^2$, $w, u \in B_r(X)$
such that $w_1=u_1$ and $w_r= u_r$ and all $m =
\max\{|P(X)|,|F(X)|\}$, $C_m(w) = C_m(u)$.  This reduces to the
problem of deciding for all words $z$ of length at most $|C(X)|^2 +
2\max\{|P(X)|,|F(X)|\}$, whether $z \in B(X)$. This is accomplished
by using a labelled graph presentation of a sofic
shift~\cite[Section 3.4]{LM}; such a presentation can be used to
give a bound on $|F(X)|$, $|P(X)|$, and therefore also on $|C(X)|$
by a bound given implicitly at the end of the proof of
Proposition~\ref{soficCX}.
\end{proof}

For any sofic shift $X$, we can endow $C(X) \cup \{\star\}$, where $\star$ is
an extra symbol, with a semigroup structure:
%a well-defined semigroup
%structure on $C(X)$:
\begin{equation*}
C(w)C(u)= C(wu) \text{ if } wu \in B(X)
\text{ and } \star \text{ otherwise }
\end{equation*}
One can show that the multiplication is well-defined and formulate an algorithm
in terms of the semigroup to decide if $X$ is a TMF.
%One can use the ideas in the proof of
%Lemma~\ref{lemma:boundwordsize} to show that this is well-defined.
%
This is consistent with the spirit in which sofic shifts were originally
defined~\cite{Weiss}.

\begin{theorem}\label{thm:supp_TMF_is_SFT}
Let $X$ be a shift space. The following are equivalent.
\begin{enumerate}[(a)]
\item{\label{cond:MC} $X$ is the support of a stationary Markov chain.}
\item{\label{cond:sup_MRF} $X$ is the support of a stationary MRF.}
\item{\label{cond:TMF} $X$ is non-wandering and a TMF.}
\item{\label{cond:prod_TMF} $X \times X$ is non-wandering and $X$ is a TMF.}
\item{\label{cond:irred_TMC}$X$ is a TMC consisting of a finite union of
irreducible TMC's with disjoint alphabets.}
\end{enumerate}
\end{theorem}
\begin{proof}
(\ref{cond:MC}) implies (\ref{cond:sup_MRF}) by Proposition~\ref{MC_is_MRF}, and (\ref{cond:sup_MRF}) implies (\ref{cond:TMF})
follows from Lemma~\ref{lem:supp_MRF_TMF} and the fact, mentioned
above, that the support of a stationary probability measure is
non-wandering.

For (\ref{cond:TMF}) implies (\ref{cond:prod_TMF}), first observe that,
by Proposition~\ref{prop:TMF_sofic}, $X$ is sofic, and then by
Lemma~\ref{lem:sofic_non-wandering}, $X$ is non-wandering iff
$X$ has dense periodic points. Then (\ref{cond:prod_TMF})
follows since $X$ has dense periodic points iff $X\times X$ has
dense periodic points.

We now show that (\ref{cond:irred_TMC}) implies (\ref{cond:MC}).
Any irreducible TMC is the support of a stationary Markov chain,
defined by an irreducible (stochastic) probability transition matrix
$P$ such that $P_{xy} > 0$ iff
$xy \in \B_2(X)$.

Finally, we show that (\ref{cond:prod_TMF}) implies (\ref{cond:irred_TMC}).
Let $u, w \in B(X)$ and $v \in \Sigma$ such that $uv \in B(X)$
and $vw \in B(X)$. To prove that $X$ is a TMC it is sufficient to show that $uvw \in B(X)$.

There exists $\tilde{w} \in B(X)$ with $|w|=|\tilde{w}|$
and $uv\tilde{w} \in B(X)$. Also, there exists $\tilde{u}
\in B(X)$ such that $|\tilde{u}|=|u|$ and $\tilde{u}vw \in
B(X)$. Since $X \times X$ is non-wandering, there are
words $x,y \in B(X)$ with $|x|=|y|$ and
$\tilde{u}vwx\tilde{u}vw,uv\tilde{w}yuv\tilde{w} \in
B(X)$. Since $X$ is a TMF, $uvwx\tilde{u}v\tilde{w} \in
B(X)$. In particular, $uvw \in B(X)$, and so $X$
is a TMC. It is easy to see that if $X \times X$ is non-wandering,
so is $X$. Thus, $X$ is a non-wandering TMC. As mentioned in
Section~\ref{symbolic}, it follows that $X$ is a finite union of
irreducible TMC's with disjoint alphabets.
\end{proof}

We remark that in each of (\ref{cond:TMF}) and
(\ref{cond:prod_TMF}), TMF can be replaced with TMC. Also,
%bhm
as mentioned above, a shift space $X$ is nonwandering if $X \times
X$ is nonwandering. The converse is true for TMF's due to the
equivalence of (\ref{cond:TMF}) and (\ref{cond:prod_TMF}); in fact,
the converse is true more generally for sofic shifts (this follows
from Lemma~\ref{lem:sofic_non-wandering}). However, the converse is
false for shift
spaces in general. %and the equivalence of (\ref{cond:sup_MRF})
%and (\ref{cond:irred_TMC}) implies that the support of a
%stationary MRF is a TMC.

\section{Stationary MRF's are Markov chains}

\begin{theorem}
\label{MRF_MC}
Let $\mu$ be a stationary measure on $\Sigma^\mathbb{Z}$. Then
$\mu$ is an MRF iff it is a Markov chain.
\end{theorem}
\begin{proof}

By Proposition~\ref{MC_is_MRF}, every  Markov chain is an MRF.

For the converse, let $\mu$ be an MRF.  By Theorem
\ref{thm:supp_TMF_is_SFT}, $supp(\mu)$ is a finite union of
irreducible TMC's on disjoint alphabets.  Therefore $\mu$ is a convex
combination of MRF's supported on irreducible TMC's. So, it suffices
to assume that $X= supp(\mu)$ is an irreducible TMC.

% then $\mu$ is a Markov chain;
%a convex combination of
If $X$ were a full shift,
then it would be a Markov chain by \cite[Corollary 3.9]{Georgii}.  We
will reduce to this case.

Let $p$ be the period and $t$ be the index of primitivity of $X$.
Let $\Sigma_0,\Sigma_1, \ldots, \Sigma_{p-1}$ be the
cyclically moving subsets of $X$. For $r \in \N$, a multiple of $p$, let
\begin{equation*}
B^0_r(X)=\{a_{-r}\ldots a_{-1} \in B_r(X)\; : \;a_{-r} \in \Sigma_0\}.
\end{equation*}
Fix  $x_{-r}\ldots x_{-1} \in B_{r}^0(X)$, $x_0 \in \Sigma_0$ and
$L$ a multiple of $p$ such that $L > r + tp$. Fix $i\in \N$. Then,
with the summations below taken over all values of $a_{iL-r} \ldots
a_{iL-1} \in B^0_r(X)$, we have
$$
\mu(x_0 | x_{-1}, \ldots, x_{-r}) =  \sum\mu(x_0, a_{iL-r},
\ldots, a_{iL-1}\; |\; x_{-1}, \ldots, x_{-r} )
$$
$$
= \sum \mu(x_0\;|\;x_{-1}, \ldots, x_{-r}, a_{iL-r},
\ldots, a_{iL-1}  ) \mu( a_{iL-r}, \ldots, a_{iL-1}
\;|\;x_{-1}, \ldots, x_{-r})
$$
$$
=\sum \mu(x_0\;|\; x_{-1}, a_{iL-r})\mu( a_{iL-r},
\ldots,  a_{iL-1}\;|\; x_{-1}, \ldots, x_{-r})
$$
(the last equality follows from the MRF property).

Thus,
\begin{equation}
\label{decomp}
\mu(x_0 | x_{-1}, \ldots, x_{-r}) = \sum \mu(x_0\;|\; x_{-1}, a_{iL-r})\mu( a_{iL-r} , \ldots,  a_{iL-1} \;|\; x_{-1}, \ldots, x_{-r}).
\end{equation}

Let
$$
X^0 = \{x \in X: x_0 \in \Sigma_0\}
$$
and let
\begin{equation*}
\phi: X^0\longrightarrow \{B_r^0(X)\}^\Z
\end{equation*}
be given by
\begin{equation*}
(\phi(x))_i= (x_{iL-r} \ldots x_{iL-1}).
\end{equation*}
Let $\mu^\p$ be the probability measure on $\{B_r^0(X)\}^\Z$ given
by the push-forward of the
%bhm
measure $\mu$ by $\phi$, i.e., $\mu^\p(U)= \mu(\phi^{-1}(U))$.

Observe that $\supp(\mu^\p)= \{B_r^0(X)\}^\Z$ by definition of $L$
and $t$. We show that $\mu^\p$ is a stationary MRF with alphabet
$B_r^0(X)$, as follows.   Let $x \in X_0$.  Below, we write
$x_{iL-r} \ldots x_{iL-1}$ as $x_{iL - r}^{iL - 1}$ when viewed as a
word of length $r$ and as $b_i$ when viewed as a a single symbol in
the alphabet, $B_r^0(X)$, of $\mu^\p$. Then

$$
\mu'([b_0,\ldots,b_n]_0 \mid [b_{-N},\ldots,b_{-1}]_{-N}\cap [b_{n+1},\ldots b_{n + M}]_{n+1})
$$
$$
=
\mu(x_{-r}^{-1}, \ldots,  x_{nL- r}^{nL -1} ~|~
x_{-NL-r}^{-NL -1}, \ldots,  x_{-L -r}^{-L -1},
x_{(n+1)L - r}^{(n+1)L -1},  \ldots,  x_{(n+M)L - r}^{(n+M)L-1})
$$
$$
= \mu(x_{-r}^{-1}, \ldots,  x_{nL- r}^{nL -1}   ~|~
x_{-L-1},
x_{(n+1)L-r})
$$
$$
= \mu(x_{-r}^{-1}, \ldots,  x_{nL- r}^{nL -1}    ~|~
 x_{-L-r}^{-L-1}, x_{ (n+1)L  -r}^{(n+1)L -1})
$$
$$
=
\mu'([b_0,\ldots,b_n]_0 \mid [b_{-1}]_{-1}\cap [b_{n+1}]_{n+1}).
$$
Thus,  $\mu^\p$ is a stationary MRF with full support.

By~\cite[Corollary 3.9]{Georgii}, $\mu^\p$ is a fully supported
stationary Markov chain, and thus has a positive stationary
distribution, which we denote by $\pi$.

Therefore, for any $a_{-r} \ldots a_{-1} \in B_r^0(X)$, we have
$$
\lim_{i \rightarrow \infty}\mu( a_{iL-r} = a_{-r}, \ldots, a_{iL-1}
= a_{-1} \;|\;x_{-1}, \ldots, x_{-r})
$$
$$
=\lim_{i \rightarrow \infty}\mu^\p(b_i=  a_{-r}  \ldots a_{-1} \;|\; b_0=x_{-1}\ldots x_{-r})=\pi(a_{-r} \ldots  a_{-1}).
$$

By compactness of $[0,1]$, there is a sequence $i_k$ such that
\begin{equation*}
\lim _{k \rightarrow \infty} \mu(x_0\;|\; x_{-1}, a_{i_kL-r})
\end{equation*}
exists.
Returning to (\ref{decomp}), we obtain
$$
\mu(x_0\;|\;x_{-1}, x_{-2}, \ldots, x_{-r})=\lim_{k \rightarrow \infty}\sum \mu(x_0\;|\; x_{-1}, a_{i_kL-r} )
\mu( a_{i_kL-r},  \ldots, a_{i_kL-1} \;|\;x_{-1}, x_{-2}, \ldots, x_{-r})
$$
$$
=\sum_{a_{-r} \ldots a_{-1} \in B_r^0(X)} \pi(a_{-r}
\ldots a_{-1}) \lim_{k \rightarrow \infty } \mu(x_0\;|\; x_{-1},
a_{i_kL-r}),
$$
which does not depend on any of $x_{-2}, x_{-3}, \ldots, x_{-r}$ . Therefore,
\begin{equation*}
\mu(x_0\;|\;x_{-1}, x_{-2}, \ldots, x_{-r})=\mu(x_0\;|\;x_{-1}).
\end{equation*}
Since the only restriction on $r$ is that it is a multiple of $p$, it
can be chosen arbitrarily large and so
$\mu$ is a stationary Markov chain, as desired.
\end{proof}

As noted in the course of the proof of Theorem~\ref{MRF_MC}, in the
case where the support is a full shift, the result is well known. It
can also be inferred from~\cite[Theorems 10.25, 10.35]{Georgii} in
the case where $\mu$ satisfies a mixing condition~\cite[Definition
10.23]{Georgii} (that condition is slightly stronger than the
concept of irreducibility used in our paper; the results
in~\cite{Georgii} apply to
%bhm
certain stationary processes that may be infinitely-valued).

Closely related to the notion of MRF is the notion of Gibbs measures~\cite{Preston}.
It is an old result~\cite{Preston} that any Gibbs measure defined
by a nearest-neighbour potential is an MRF. Also it is easy to see
that any Markov chain is a Gibbs measure, defined by a
nearest-neighbour potential determined by its transition
probabilities. So, in the one-dimensional (discrete-time,
finite-valued) stationary case,  MRF's, Gibbs measures and Markov
chains are all the same (where we assume that our MRF's and Markov
chains are first order and our Gibbs measures are
nearest-neighbour).

\section{Continuous-time Markov random fields and Markov Chains}
\label{continuous}

We begin with a definition of continuous-time processes. This
presentation is compatible with standard references, for example
~\cite[Section 2.2]{norris}. To avoid ambiguity and measurability
issues, we will use the following:
\begin{defn}
A \emph{continuous-time stationary process (CTSP)} taking values in
a finite set $\Sigma$ is a translation-invariant probability measure
on the space $RC(\Sigma)$ of right-continuous functions from
$\mathbb{R}$ to $\Sigma$ with the $\sigma$-algebra $\mathcal{B}$
generated by \emph{cylinder sets} of the form:
$$ [a_1,\ldots,a_n]_{t_1,\ldots,t_n} = \{ x \in RC(\Sigma) ~:~
x_{t_i}=a_i~,~ i=1,\ldots,n\},$$ where $a_i \in \Sigma$ and $t_i \in
\mathbb{Q}$.
\end{defn}

For sets $A, B \in \B$ and Borel measurable $I \subseteq \mR$, we
will use the shorthand notation $\mu(A|B \mbox{ and } x_I)$ to mean
$\mu(A|B \cap \{y \in RC(\Sigma): ~y_I = x_I\}    )$.

\begin{defn}
A  \emph{continuous-time stationary Markov random field (CTMRF)}
taking values in a finite set $\Sigma$ is a CTSP which  satisfies:
$$
\mu\left( [a_1,\ldots,a_n]_{t_1,\ldots,t_n} ~|~ [a,b]_{s,t} \right)
=
 \mu\left([a_1,\ldots,a_n]_{t_1,\ldots,t_n} ~|~
 [a,b]_{s,t} \cap [b_1,\ldots,b_m]_{s_1,\ldots,s_m} \right),
 $$
whenever \newline $\;t_1,\ldots,t_n \in(s,t)$, $s_1,\ldots,s_m \in
[s,t]^c$ and $\mu([a,b]_{s,t} \cap
[b_1,\ldots,b_m]_{s_1,\ldots,s_m}) > 0$
\end{defn}

Informally, a CTMRF can be viewed as a stationary process such that
%bhm
for all $s < t$, the distribution of $(x_u)_{u\in (s,t)}$ is
independent of $(x_u)_{u \in [s,t]^c}$ given $\{x_s,x_t\}$.

\begin{defn}
A \emph{continuous-time stationary Markov Chain (CTMC)} taking
values in a finite set  $\Sigma$ is a CTSP which satisfies
$$\mu\left( [a_1,\ldots,a_n]_{t_1,\ldots,t_n} ~|~ [a]_{s} \right) =
\mu\left([a_1,\ldots,a_n]_{t_1,\ldots,t_n} ~|~ [a]_{s} \cap
[b_1,\ldots,b_m]_{s_1,\ldots,s_m} \right),$$ whenever
$t_1,\ldots,t_n  > s $, $s_1,\ldots,s_m  < s$ and $\mu([a]_{s} \cap
[b_1,\ldots,b_m]_{s_1,\ldots,s_m}) > 0$.
\end{defn}

In this section, we prove:
\begin{prop}
\label{cts-time}Any continuous-time (finite-valued) stationary ergodic
Markov random field is a continuous-time stationary Markov chain.
\end{prop}

Note that in this statement, there are no positivity assumptions on
the conditional probabilities
$\mu\left( [a_1,\ldots,a_n]_{t_1,\ldots,t_n} ~|~ [a,b]_{s,t} \right)$.

We prove this result by reducing it to the discrete-time case.

\begin{proof}
Let $\mu$ be a CTMRF.
%bhm
It suffices to show that for any $n \in \N$  and $t<t_1<t_2<\ldots <
t_n \in \R$
\begin{equation*}
\mu(x_t\;|\; x_{t_1}, x_{t_2}, \ldots, x_{t_n})= \mu(x_t\;|\; x_{t_1}).
\end{equation*}
Let $\epsilon>0$. By the right continuity of the elements of
$RC(\Sigma)$, there exists $m \in \N$ such that
\begin{eqnarray}
\label{eqn 1}|\mu(x_t\;|\; x_{t_1}, x_{t_2}, \ldots, x_{t_n})-\mu (x_{\tilde{t}}\;|\;x_{\tilde{{t}}_1}, x_{\tilde{{t}}_2}, \ldots, {{x_{\tilde{t}_n}}})|&<&\epsilon\\
\label{eqn 2}|\mu(x_t\;|\; x_{t_1} )-\mu,
(x_{\tilde{t}}\;|\;x_{\tilde{{t}}_1})|&<&\epsilon
\end{eqnarray}
where for any $r \in \mathbb{R}$, $\tilde{r}$ means $\frac{\lceil r m \rceil}{m}$.
Also, the process restricted to evenly spaced discrete points forms
an MRF; that is, defining $\phi: RC(\Sigma) \longrightarrow
\Sigma^\Z$  by $\phi(x)_i =x_{\frac{i}{m}}$, the push-forward of
$\mu$, given by $\tilde{\mu}(F)= \mu({\phi^{-1}(F)})$, is an
(discrete-time) MRF. By Theorem~\ref{MRF_MC}, $\tilde{\mu} $ is a
(discrete-time) Markov Chain. Hence
 \begin{eqnarray}
\mu (x_{\tilde{t}}\;|\;x_{\tilde{t}_1}, x_{\tilde{{t}}_2}, \ldots, {{x_{\tilde{t}_n}}})&=
& \tilde{\mu}(\phi(x)_{m\tilde{t}}\;|\; \phi(x)_{m\tilde{t}_1}, \phi(x)_{m\tilde{t}_2}, \ldots, \phi(x)_{m\tilde{t}_n})\nonumber\\
&=&\tilde{\mu}(\phi(x)_{m\tilde{t}}\;|\; \phi(x)_{m\tilde{t_1}} )\nonumber\\
&=&\mu (x_{\tilde{t}}\;|\;x_{\tilde{{t}}_1})\label{eqn 3}.
\end{eqnarray}
By (\ref{eqn 1}), (\ref{eqn 2}), (\ref{eqn 3}) and the triangle
inequality, we get
\begin{equation*}
|\mu(x_t\;|\; x_{t_1}, x_{t_2}, \ldots, x_{t_n})- \mu(x_t\;|\;
x_{t_1})|<2 \epsilon.
\end{equation*}
Since $\epsilon$ was arbitrary,
\begin{eqnarray*}
\mu(x_t\;|\; x_{t_1}, x_{t_2}, \ldots, x_{t_n})=\mu(x_t\;|\; x_{t_1}).
\end{eqnarray*}
\end{proof}

We conclude this section with an ergodic-theoretic consequence. A
{\em measure-preserving flow} is a collection of
%bhm
invertible measure preserving mappings $\{T_t\}_{t \in \mR}$ of a
%bhm
probability space such that $T_{t+s} = T_t\circ T_s$ for all $t,s$
and the map $(t,x) \mapsto T_t(x)$ is (jointly) measurable. Any CTSP
is a measure preserving flow.

A measure-preserving flow is  {\em ergodic} if the only invariant
functions are the constant functions, i.e., if for all $t$, $f\circ
T_t = f$ a.e., then $f$ is constant a.e.  And $\{T_t\}$ is {\em weak
mixing} (or {\em weakly mixing}) if it has no non-trivial eigenfunctions, i.e., if $\lambda
\in \mathbb{C}$ and for all $t$, $f\circ T_t=\lambda^tf$ a.e., then
$f$ is constant a.e.

Weak mixing is in general a much stronger condition than ergodicity.  However,
it is well known that any stationary ergodic continuous-time Markov chain is weakly
mixing \cite{KvN}.   So, a consequence of Proposition~\ref{cts-time} is:

\begin{prop}
\label{lem:CTMRF_WM} Any continuous-time (finite-valued) stationary ergodic
Markov random field is weakly mixing.
\end{prop}

Below, we give a direct proof of this result. For $I \subset
\mathbb{R}$, we use $\mathcal{B}(\Sigma^{I})$ to denote the
$\sigma$-algebra generated by cylinder sets of the form:
$$ [a_1,\ldots,a_n]_{t_1,\ldots,t_n} = \{ x \in RC(\Sigma) ~:~
x_{t_i}=a_i~,~ i=1,\ldots,n\},$$ where $a_i \in \Sigma$ and $t_i \in
\mathbb{Q}\cap I$.

\begin{proof}
Suppose $\mu$ is an ergodic CTMRF on the alphabet $\Sigma$. Suppose that
 $f$ is a non-constant $L^2(\mu)$-eigenfunction of $T_t$, with $\lambda$ the
 corresponding eigenvalue:
$f(T_t(x))=\lambda^tf(x)$. By the assumption that $T_t$ is ergodic,
$\lambda \ne 1$. By normalizing, we can assume $|f|=1$ a.e., thus
$\|f\|_2=1$. For any $\epsilon>0$, there is a sufficiently large $n$
and a $\mathcal{B}(\Sigma^{(-n,n)})$-measurable function $f_n$ with
$\|f_n-f\|_2 \le \epsilon$.

It follows that for any $t \in \mathbb{R}$,
$\|f_n-T_t\lambda^{-t}f_n\|_2 \le 2\epsilon$. Now denote by
$\hat{f_n}$ the conditional expectation of $f_n$ with respect
to $\mathcal{B}(\Sigma^{\{-n,n\}})$. By the MRF property of $\mu$,
since $f_n$ is  $\mathcal{B}(\Sigma^{(-n,n)})$-measurable and
$T_t\lambda^{-t}f_n$ is
$\mathcal{B}(\Sigma^{(-n+t,n+t)})$-measurable it follows that for $t
> 2n$,
$$
\int (f_n) (T_t\lambda^{-t}f_n) d\mu = \int (\hat{f_n})
 (T_t \lambda^{-t}\hat{f_n}) d\mu.
$$
It follows that
\begin{equation}\label{eq:approx_eigen}
\|\hat{f_n}-T_t\lambda^{-t}\hat{f_n}\|_2 \le 2\epsilon.
\end{equation}

We claim that for sufficiently small $\epsilon$ it is impossible for
\eqref{eq:approx_eigen} to hold simultaneously for all sufficiently
large $t$.   To see this, first note that due to the MRF property, $\hat{f_n}$ takes at most
$|\Sigma|^2$ values.  Thus,  there is some $c \in \mathbb{C}$ with
$|c|=1$ such that for all $x$, $|\hat{f_n}(x)-c| >
\frac{1}{2|\Sigma|^2}$, and some $d \in \mathbb{C}$ with $|d|=1$
with $\mu(\hat{f_n}(x)=d) \ge \frac{1}{|\Sigma|^2}$. Now take $t >
2n$ such that $\lambda^td =c$.

Now for any $x$ such that $\hat{f}^n(x) = d$,  let $z =
T_t\hat{f}_n(x)$. Then
$$
|\hat{f}_n(x) - \lambda^{-t} T_t\hat{f}_n(x)| = |d - \lambda^{-t}z|
= |\lambda^t d - z| = |c - z| \ge  \frac{1}{2|\Sigma|^2},
$$
and so $\|\hat{f_n}-T_t\lambda^{-t}\hat{f_n}\|_2 \ge
\frac{1}{2|\Sigma|^3}$. This contradicts \eqref{eq:approx_eigen} for
$\epsilon < \frac{1}{4|\Sigma|^3}$
\end{proof}

\section{MRF's in Higher Dimensions}

There is a vast literature on Markov random fields and Gibbs
measures in higher dimensions. Most of the concepts in this paper
can be generalized to higher dimensional processes.  For instance,
we can define a $\Z^d$ TMF as a $\Z^d$ shift space $X$ such that
whenever $A$ and $B$ are finite subsets of $\Z^d$ such that
$\partial A \subset B \subset A^c$, and $x, y \in X$ such that
$x_{\partial A} = y_{\partial A}$, then there exists $z \in X$
such that $z_A = x_A$ and  $z_B = y_B$. With this definition,
the proof of Lemma~\ref{lem:supp_MRF_TMF} carries over to show
that the support of a $\Z^d$ stationary MRF is a $\Z^d$ TMF. However, most of
the other results in this paper fail in higher dimensions. For
example, there are $\Z^d$ TMF's which are not even sofic.  And there
are stationary $\Z^d$ MRF's that are not Gibbs measures. However,
there are some positive things that can be said, and this is a topic
of ongoing work.

\bibliographystyle{abbrv}

\end{document}